\documentclass[11pt]{amsart}
\usepackage{amssymb}
\usepackage{color}

\newtheorem{theorem}{Theorem}
\newtheorem{lemma}[theorem]{Lemma}
\newtheorem{proposition}[theorem]{Proposition}

\newtheorem{corollary}[theorem]{Corollary}

\theoremstyle{definition}

\newtheorem{remark}[theorem]{Remark}
\newtheorem{example}[theorem]{Example}

\newcommand{\D}{\mathcal D}
\renewcommand{\H}{\mathcal H}
\newcommand{\F}{\mathcal F}
\newcommand{\U}{\mathcal U}
\renewcommand{\int}{{\rm int}}

\newcommand{\R}{\mathbb R}

\newcommand{\Z}{\mathbb Z}
\newcommand{\N}{\mathbb N}
\newcommand{\Q}{\mathbb Q}

\def\la{\langle}
\def\ra{\rangle}

\newcommand{\co}{\mathfrak c}
\newcommand{\ES}{\mathrm{ES}}
\newcommand{\FC}{\mathrm{FC}}
\newcommand{\VFC}{\mathrm{VFC}}
\newcommand{\TFC}{\mathrm{TFC}}

\author{Marek Balcerzak}
\address{Institute of Mathematics, Lodz University of Technology, W\'olcza\'nska 215, 90-924
\L\'od\'z, Poland}
\email{marek.balcerzak@p.lodz.pl}
\author{Tomasz Natkaniec}
\address{Institute of Mathematics,
Faculty of Mathematics, Physics and Informatics,
University of Gda\'nsk,
Wita Stwosza 57, 80-952 Gda\'nsk, Poland}
\email{mattn@mat.ug.edu.pl}
\author{Ma\l gorzata Terepeta}
\address{Center of Mathematics and Physics,  al. Politechniki 11,
and Institute of Mathematics, W\'olcza\'nska 215, Lodz University of Technology,   90-924
	\L\'od\'z, Poland}
\email {malgorzata.terepeta@p.lodz.pl}

\title{Families of feebly continuous functions and their properties}
\date{}

\begin{document}

\begin{abstract}
Let $f\colon\R^2\to\R$. The notions of feebly continuity and very feebly continuity of $f$ at a point $\la x,y\ra\in\R^2$
were considered by I. Leader in 2009. We study properties of the sets $FC(f)$ (respectively, $VFC(f)\supset FC(f)$)
 of points at which $f$ is feebly continuous (very feebly continuous).
 We prove that $VFC(f)$ is densely nonmeager, and, if $f$ has the Baire property (is measurable),
 then $FC(f)$ is residual (has full outer Lebesgue measure). We describe several
 examples of functions $f$ for which $FC(f)\neq VFC(f)$. Then we consider the notion of two-feebly continuity which is strictly weaker than very feebly continuity. We prove that the set of points where (an arbitrary)
 $f$ is two-feebly continuous forms a residual set of full outer measure.
 Finally, we study the existence of large algebraic structures inside or outside various sets of feebly continuous functions.
\end{abstract}
\maketitle
\section{Introduction}
The notion of feebly continuous real-valued functions has been recently considered by Dales and Leader
(cf. \cite{L}) in the case where the domain is $\R$ or $\R^2$. Feebly continuity was known earlier since the respective definition appeared in Thomson's book \cite[p. 70]{T}. Its idea arose probably from the theory of cluster sets for 
arbitrary functions (cf. \cite{C}).

Following \cite {T}, a function $f\colon \R\to\R$ is called {\em feebly continuous at a point} $x\in\R$ if there is a sequence $x_n\to x$
with terms different from $x$ such that $f(x_n)\to f(x)$. Clearly, we may assume that $(x_n)$ is monotone.
Leader in \cite{L} assumed that $(x_n)$ is strictly decreasing, that is $x_n\searrow x$.
For simplicity, we will consider this version. So, we say that $f$ is {\em feebly continuous} at $x$ whenever there exists
a sequence $x_n\searrow x$ such that $f(x_n)\to f(x)$.

It was shown in \cite{L} that, for every function $f\colon\R\to\R$, the set of points of at which $f$ is feebly continuous
is co-countable. Let us observe that, given a countable infinite set $A\subset\R$, there is a function $f\colon \R\to\R$ such that $f$ fails to be feebly continuous exactly at points in $A$. Namely, let $f(x):=0$ for $x\in\R\setminus A$, and $f(a_n):=n$ whenever $(a_n)$ is 
a one-to-one enumeration of $A$. Note that we  can modify $f$ to produce a function that is extremally bad, for instance, nonmeasurable. Indeed, we pick a nonmeasurable
set $B\subset\R\setminus A$ and put $\widetilde{f}:=f-\chi_{B}$ where $f$ is as before. (As usual, $\chi_B$ stands for the characteristic function of $B$.)

The case where a function is defined on the plane is more interesting. According to \cite{L}, we say that
$f\colon\R^2\to\R$ is {\em feebly continuous at a point} $\la x,y\ra\in\R^2$  if there exist sequences $x_n\searrow x$ and $y_m\searrow y$ such that
$$
\lim_{n\to\infty}\lim_{m\to\infty}f(x_n,y_m)=f(x,y).
$$
Surprisingly, we have observed that this notion depends on the order of limits -- the respective example will be presented below.
Hence, if we change the order of limits in the above definition as follows,
$$
\lim_{m\to\infty}\lim_{n\to\infty}f(x_n,y_m)=f(x,y),
$$
we obtain a different notion; then $f$ will be called {\em reverse feebly continuous} at $\la x,y\ra$.

\begin{example} \label{exa1}
Let $f\colon\R^2\to\R$ be given by
\begin{displaymath}
f(x,y):=\left\{ \begin{array}{ll}
\frac{x\sin(1/x)+2y}{x+y} & \textrm{if $x+y\neq 0$}\\
2 & \textrm{otherwise.}
\end{array} \right.
\end{displaymath}
Let $x_n\searrow 0$ and $y_m\searrow 0$. Then $x_n+y_m\neq 0$ for any $m$ and $n$. Hence 
$\lim_{m\to\infty}f(x_n,y_m)=\sin(1/x_n)$
and $\lim_{n\to\infty}\sin(1/x_n)$ either does not exist or is between $-1$ and $1$.
Consequently, $f$ is not feebly continuous at $\la 0,0\ra$.
On the other hand, for any sequences $x_n\searrow 0$ and $y_m\searrow 0$, we have
$\lim_{n\to\infty}f(x_n,y_m)=2$ for every $m$. Hence
$$\lim_{m\to\infty}\lim_{n\to\infty}f(x_n,y_m)=2=f(0,0)$$
which shows that $f$ is reverse feebly continuous at $\la 0,0\ra$.
\end{example}

\begin{proposition} \label{wst}
If $f\colon\R^2\to\R$ is continuous at $\la x,y\ra$ then $f$ is feebly continuous and reverse feebly continuous at
$\la x,y\ra$.
\end{proposition}
\begin{proof}
Choose a sequence of open boxes $(I_n\times J_n)$  and a sequence $x_n\searrow x$ with the following properties
for each $n\in\N$: 
\begin{itemize}
\item
$\la x,y\ra\in I_n\times J_n$; 
\item
$I_m\times J_m\subset I_n\times J_n$ whenever $m>n$;
\item
$|f(s,t)-f(x,y)|<\frac{1}{n}$ for $\la s,t\ra\in I_n\times J_n$;
\item
$x_n\in I_n$.
\end{itemize}
Fix a sequence $(y_m^1)$ with terms in $J_1$ such that $f(x_1,y^1_m)$ converges to some $\alpha_1$. This is possible because the set $f[\{ x_1\}\times J_1]$ is bounded. Note that $|\alpha_1- f(x,y)|\le 1$. Proceeding by induction, for every $n\in\N$ we choose a sequence $(y^n_m)_m$ with terms in $J_n$ such that 
\begin{itemize}
\item
$(y^{n+1}_m)_m$ is a subsequence of $(y^n_m)_m$;
\item
$(f(x_n,y^n_m))_m$ converges to some $\alpha_n\in [f(x,y)-\frac{1}{n}, f(x,y)+\frac{1}{n}]$.
\end{itemize}
Finally, put $y_m:=y^m_m$ for every $m\in\mathbb{N}$. Then for every $n$, the sequence $(y_m)_{m>n}$ is a subsequence of $(y^n_m)_m$. Hence $\lim_m f(x_n,y_m)=\alpha_n$, and consequently, $\lim_n\lim_m f(x_n,y_m)=\lim_n\alpha_n=f(x,y)$. Thus $f$ is feebly continuous at $\la x,y\ra$.
 Similarly, we show that $f$ is reverse feebly continuous at $\la x,y\ra$.
\end{proof}

Easy examples show that the implication in Proposition \ref{wst} cannot be reversed in general.

\section{Points of feebly continuity and very feebly continuity}

In \cite{L}, Leader considered another notion which is weaker than feeble continuity for functions defined on the plane.
Namely, $f\colon\R^2\to\R$ is called {\em very feebly continuous at} a point $\la x,y\ra$ if there exist
a sequence $x_n\searrow x$ and, for each $n\in\N$, a sequence $y_m^{(n)}\searrow y$ such that
\begin{equation} \label{eq3}
\lim_{n\to\infty}\lim_{m\to\infty}f(x_n,y_m^{(n)})=f(x,y).
\end{equation}
It was proved in \cite{L} that every function $f\colon\R^2\to\R$ has a point of very feebly continuity.
By a similar reasoning, we will prove a sharper version of this result.

For $f\colon \R^2\to \R$, denote by $FC(f)$ (respectively, $VFC(f)$), the sets of feebly continuity (very feebly continuity)
points of the function $f$. Plainly, $FC(f)\subset VFC(f)$.
The following lemma is a direct consequence of the definition and formula (\ref{eq3}).

\begin{lemma} \label{le1} 
Let $f\colon \R^2\to \R$. A point $z=\la x,y\ra$ belongs to $\R^2\setminus VFC(f)$ if and only if
there exist an interval $(p,q)$ containing $f(z)$,
a real $t>0$ and real numbers $r_s>0$, chosen for every $s\in(0,t)$, such that $f$ does not 
attain values in $(p,q)$ at any point of the set
$$G(z):=\{\la x+a,y+b\ra\colon 0<a<t,\; 0<b<r_a\} .$$
\end{lemma}

\begin{theorem} \label{PP1}
For any function $f\colon\R^2\to\R$, the set VFC(f) is densely nonmeager,
that is, its intersection with every box $(\alpha,\beta)\times(\gamma,\delta)$ is nonmeager.
\end{theorem}
\begin{proof}
Suppose that $f$ is not very feebly continuous at a point $z=\la x,y\ra\in\R^2$.
Then pick an interval $(p,q)$ and the respective set $G(z)$, as in Lemma \ref{le1}.
We can assume that $p,q$ are rational.
In this case, we say that $z$ is of type $(p,q)$. We denote by $S_{p,q}$ the set of all points of type $(p,q)$.

Fix any box $B:=(\alpha,\beta)\times(\gamma,\delta)$.
For $E\subset\R^2$ and $x\in\R$, let $E_x:=\{y\in\R\colon \la x,y\ra\in E\}$ denote the respective vertical section of $E$.
Suppose that $H:=VFC(f)\cap B$ is meager. 
By the Kuratowski-Ulam theorem, we can find a meager set $M\subset\R$ of type $F_\sigma$ such that
$H_x$ is meager for each $x\in\R\setminus M$.  Fix $x\in (\alpha,\beta)\setminus M$. Then
$$(\gamma,\delta)\setminus H_x=\bigcup_{p,q\in\Q;\; p<q} (\gamma,\delta)\cap(S_{p,q}\setminus H)_x .$$
Since $(\gamma,\delta)\setminus H_x$ is non-meager, there are $p,q\in\Q$ for which the set $(\gamma,\delta)\cap (S_{p,q}\setminus H)_x$ is non-meager.
So, there is an interval $(\gamma',\delta')\subset (\gamma,\delta)$
with rational endpoints such that
$(S_{p,q}\setminus H)_x$ is dense in $(\gamma',\delta')\setminus H_x$. Now, we have
$$(\alpha,\beta)\setminus M=\bigcup_{p,q\in\Q;\; p<q}\;\; \bigcup_{u,v\in\Q;\; u<v} T_{p,q,u,v}$$
where 
$$T_{p,q,u,v}:=\{ x\in(\alpha,\beta)\setminus M\colon (S_{p,q}\setminus H)_x\mbox{ is dense in }(u,v)\setminus H_x\} .$$
Since the set $(\alpha,\beta)\setminus M$ is non-meager,
we find an interval $(\alpha',\beta')\subset (\alpha,\beta)$ and parameters $p,q,u,v\in\Q$ 
such that $T_{p,q,u,v}$ is dense in $(\alpha',\beta')\setminus M$. Fix any point 
$$z\in ((\alpha',\beta')\setminus M)\times (u,v))\cap S_{p,q}\setminus H\mbox{ 
with }z=\la x,y\ra.$$
Let $t$ be as in the definition of $G(z)$ where $G(z)$ is as above. Fix $x'\in (x,x+t)\cap T_{p,q,u,v}$. 
Next, choose $y'\in(u,v)\cap (G(z))_{x'}\cap(S_{p,q}\setminus H)_{x'}$.
Since $\la x',y'\ra\in S_{p,q}$, we have $f(x',y')\in (p,q)$. On the other hand, $\la x',y'\ra\in G(z)$, so $f(x',y')\notin (p,q)$.
Contradiction.
\end{proof}

In a contrast to Theorem \ref{PP1}, Leader in \cite{L} constructed, under {\bf CH},
an example of a function $f$, from $\R^2$ onto $\Z$, which is nowhere 
feebly continuous. We will prove that this is impossible for measurable functions and functions with the Baire property.
Moreover, in those cases, the set $FC(f)$ is large.

%Recall that open sets modulo meager sets are said to have the Baire property. They form a $\sigma$-algebra,
%and real valued functions measurable with respect to it are said to have the Baire property (see \cite{O}).

\begin{theorem} \label{TT1}
Let $f\colon\R^2\to\R$.
\begin{enumerate}
\item If $f$ has the Baire property then the intersection of $FC(f)$ with every nonmeager set $B\subset\R^2$ having the
Baire property contains a product of two perfect subsets of $\R$. Consequently, $FC(f)$ is residual.
\item If $f$ is Lebesgue measurable then the intersection of $FC(f)$ with every measurable set $B\subset\R^2$ of positive measure contains a product of two perfect subsets of $\R$.
Consequently, $FC(f)$ has full outer Lebesgue measure. 
\end{enumerate}
\end{theorem}
\begin{proof}
(1) Let $f\colon\R^2\to\R$ have the Baire property and let $B\subset\R^2$ be nonmeager with the Baire property. Pick a residual set $H\subset\R^2$ such that $f|H$ is continuous.
 By the Galvin theorem \cite[Theorem 19.6]{Ke}, pick Cantor-type sets $P,Q\subset\R$ such that $P\times Q\subset H\cap B$. 
Let $\la x,y\ra \in P\times Q$ be a point such that $x$ is a right-hand accumulation point of $P$ and $y$ is a right-hand accumulation point of $Q$. Pick $x_n\searrow x$ and $y_m\searrow y$ where
$x_n$'s are in $P$ and $y_m$'s are in $Q$. Since $f|H$ is continuous at 
each $\la x_n,y\ra$, we have $\lim_{m\to\infty}f(x_n,y_m)=f(x_n,y)$. Now, by the continuity of $f|H$ at 
$\la x,y\ra$, 
$$
\lim_{n\to\infty}\lim_{m\to\infty}f(x_n,y_m)=f(x,y)$$
which shows that $f$ is feebly continuous at $\la x,y\ra$. 
Let $P_0$ (respectively, $Q_0$) denote the set of all right-hand isolated points of $P$ (respectively, $Q$). Then $P_0$ and $Q_0$ are countable, so we can find perfect sets $P_1\subset P\setminus P_0$ and $Q_1\subset Q\setminus Q_0$.  
Hence $P_1\times Q_1\subset FC(f)\cap B$. Supposing that $\R^2\setminus FC(f)$ is nonmeager, we obtain a contradiction with 
our first assertion.

(2) The proof for measurable $f$ is similar. Let $B\subset\R^2$ be measurable of measure $\alpha\in(0,\infty)$. By the Luzin theorem, pick
a measurable set $M\subset B$ such that the measure of $B\setminus M$ is $<\alpha/2$ and $f|M$ is continuous.
Then by a theorem of Brodskii \cite{Br} (cf. also \cite{E}), pick Cantor-type sets $P,Q$ such that $P\times Q\subset M$.
The remaining argument works as above. To show the final assertion, suppose that $FC(f)$ is not of full outer measure.
Then $\R^2\setminus FC(f)$ has positive inner measure, hence it  contains a measurable set of  positive measure
and we obtain a contradiction.
\end{proof}

\begin{corollary}\label{Borel}
If $f\colon\R^2\to\R$ is a Borel function then the sets $FC(f)$ and $VFC(f)$ are analytic residual of full Lebesgue measure.
\end{corollary}

\begin{proof}
Let $f$ be Borel measurable. By Lemma~\ref{le1}, a point $\la x,y\ra$ is in $\R^2\setminus VFC(f)$ if and only if
\begin{eqnarray*}
(\exists p,q\in\Q;\, p<q) (\la x,y\ra\in f^{-1}[(p,q)]\,\land\,\\
\exists t\in\Q^+\,\forall s\in(0,t)\, \exists r_s\in\Q^+\, \forall v\in (0,r_s)\; \langle x+s,y+v\rangle\in f^{-1}[\R\setminus (p,q)].
\end{eqnarray*}
Hence it follows that $\R^2\setminus VFC(f)$ is coanalytic (cf. \cite[Section 4.1]{S}). Thus $VFC(f)$ is analytic.

The argument for $FC(f)$ is different. We simply use the definition of $\la x,y\ra\in FC(f)$.
Note that the set $c_0^-$ of strictly decreasing sequences in the separable Banach space $c_0$ is
of type $G_\delta$, so it forms a Polish space. The statement $\la x,y\ra\in FC(f)$ is equivalent to the formula
$$
\exists (x_n)\in c_0^-\; \exists (y_m)\in c_0^-\; \exists\alpha\in\R^\N\;\;\Phi(x,y,(x_n),(y_m),\alpha)
$$
where $\Phi(x,y,(x_n),(y_m),\alpha)$, given by
$$
\left(\forall n\in\N\; \lim_{m\to\infty}f(x+x_n,y+y_m)=\alpha(n)\right)\,\land\,\lim_{n\to\infty}\alpha(n)=f(x,y),
$$
describes a Borel subset of $\R^2\times(c_0^-)^2\times\R^\N$.
Hence $FC(f)$ is analytic. 

Consequently, the sets $FC(f)$ and $VFC(f)$ are measurable with the Baire property. Now, by Theorem~\ref{TT1}, the set $FC(f)$ is residual of full measure, hence its superset $VFC(f)$ has the same properties.
\end{proof}

\begin{example} \label{E3}
In general, sets $FC(f)$ and $VFC(f)$ may be without the Baire property or nonmeasurable. Let $A\subset \R^2$ be a set that intersects every $G_\delta$ nonmeager plane set, and such that no three points of $A$ are colinear (cf. \cite[Theorem 15.5]{O}).  Consider $f:=\chi_A$. Then $VFC(f)=\R^2\setminus A$.
Indeed, consider $\la x,y\ra\in\R^2$, a sequence $x_n\searrow x$ and, for each $n\in\N$, a sequence $y_m^{(n)}\searrow y$.
For each $n$, we have $\lim_m f(x_n,y^{(n)}_m)=0$ since $|\{y\colon f(x_n,y)=1\}|\leq 2$. Then 
$$
\lim_{n\to\infty}\lim_{m\to\infty}f(x_n,y_m^{(n)})=0.$$
 Hence this limit is equal $f(x,y)$ if and only if $f(x,y)=0$ which is equivalent to $\la x,y\ra\notin A$.
 The same argument shows that $FC(f)=\R^2\setminus A$.
 It follows that $A$ does not possess the Baire property. Thus $\R^2\setminus A$ is neither analytic nor residual.
 Similarly, we can use nonmeasurable set $A\subset \R^2$ that intersects every closed plane set of positive measure, and such that no three points of $A$ are colinear (cf. \cite[Theorem 14.4]{O}). Then $f:=\chi_A$ is nonmeasurable and
 $FC(f)=VFC(f)=\R^2\setminus A$.
\end{example}

It is natural to ask how much the sets $VFC(f)$ and $FC(f)$ can differ. Let us start with a preliminary example.

\begin{example}\label{EX1}
We present a Baire 1 function $f\colon\R^2\to\R$ with $VFC(f)\setminus FC(f)=\{\la 0,0\ra\}$. 
Choose sequences of reals $(y^{(n)}_m)_m$, for $n\in\mathbb{N}$, with the following properties:
\begin{itemize}
\item
$y^{(n)}_m\searrow 0$ for every $n\in\mathbb{N}$;
\item
$y^{(n)}_m\ne y^{(n')}_{m'}$ for $\la n,m\ra\ne \la n',m'\ra$.
\end{itemize}
%For all $n,m\in\mathbb{N}$ define $x_n:=\frac{1}{n}$.
Fix a sequence $x_n\searrow 0$.
Let $f:=\chi_A$ where
$$A:=\{ \la x_n,y^{(n)}_m\ra \colon n,m\in\mathbb{N}\}\cup\{\la 0,0\ra\},$$ 
First, observe that $f$ is discontinuous only at points from the countable set
 $A\cup\{ \la x_n,0\ra\colon n\in\mathbb{N}\}$, hence it is of the first Baire class. 
Next, notice that $\lim_m f(x_n,y^{(n)}_m)=1$ for every 
$n\in\mathbb{N}$, thus $\lim_n\lim_m f(x_n,y^{(n)}_m)=1=f(0,0)$, so $f$ is very feebly continuous at the point $\la 0,0\ra$. 

Now, suppose that 
$\la 0,0\ra\in FC(f)$. Then there are sequences $s_n\searrow 0$, $t_m\searrow 0$ with 
$$\lim_{n\to\infty}\lim_{m\to\infty} f(s_n,t_m)=f(0,0)=1.$$
 Clearly, $(t_m)$ must be a subsequence of all sequences $(y^{(n)}_m)_m$ (for $n\in\N$), which is impossible.

Finally, fix $\la x,y\ra\ne\la 0,0\ra$. If $\la x,y\ra\not\in A\cup\{ \la x_n,0\ra\colon n\in\mathbb{N}\}$ then $f$ is continuous at the point $\la x,y\ra$, and therefore, $\la x,y\ra\in FC(f)$. If $\la x,y\ra=\la x_n,y^{(n)}_m\ra$ then $f(x,y)=1$ and there is an open neigbourhood $U$ of $\la x,y\ra$ such that $f(v,w)=0$ for any $\la v,w\ra\in U\setminus\{\la x,y\ra\}$. Hence $\la x,y\ra\not\in VFC(f)$. If $\la x,y\ra=\la x_k,0\ra$ for some $k\in\mathbb{N}$, then $f(x,y)=0$ and there is a sequence $v_n\searrow x_k$ such that $f(v_n,y)=0$ for any $n\in\N$. Clearly, this implies that $\la x,y\ra\in FC(f)$.
\end{example}

\begin{remark}
The above example cannot be improved to obtain a function $f$ which is continuous at each point of an open neighbourhood $G$ of $\la 0,0\ra$,
distinct from $\la 0,0\ra$. Indeed, in this case,
$\la 0,0\ra\in VFC(f)$ implies that $\la 0,0\ra\in FC(f)$.
To show it, assume that $G$ is an open box $I\times J$ and let $\la 0,0\ra\in VFC(f)$. 
Pick a sequence $x_n\searrow 0$ and, for each $n$, a sequence $y^{(n)}_m\searrow 0$ such that $x_n\in I$, 
$y^{(n)}_m\in J$ (for all $m,n\in\N$), and $\lim_n\lim_m f(x_n,y^{(n)}_m)=f(0,0)$. Fix any sequence $y_m\searrow 0$ with $y_m\in J$. For every $n\in\mathbb{N}$, since $f$ is continuous at $\la x_n,0\ra$, we have $\lim_m f(x_n,y_m)=\lim_m  f(x_n,y^{(n)}_m)$. Thus $\lim_n\lim_m f(x_n,y_m)=\lim_n\lim_m f(x_n,y^{(n)}_m)=f(0,0)$, and so, $\la 0,0\ra\in FC(f)$.
\end{remark}

We use an idea from Example \ref{EX1} to obtain more general results.

\begin{theorem} \label{TN}
\begin{enumerate}
\item
For every countable set $A\subset\R^2$, there exists a Baire 2 function $f\colon\R^2\to\R$ such that $A\subset VFC(f)\setminus FC(f)$.
\item
There is a Baire 1 function $f\colon\R^2\to\R$ for which the set $VFC(f)\setminus FC(f)$ is perfect.
\item
There exists a Baire 2 function $f\colon\R^2\to\R$ for which the set $VFC(f)\setminus FC(f)$ is $\co$-dense in $\R^2$
(that is, its intersection with every nonempty open set is of size $\co$).
\end{enumerate}
\end{theorem}
\begin{proof}
(1) Assume that $A$ is countable infinite and let $\la a_k,b_k\ra$ for $k\in\mathbb{N}$ be a one-to-one enumeration of points in $A$. For every $k$, choose a sequence $x^k_n\searrow a_k$ such that the sets $\{ x^k_n\colon n\in\mathbb{N}\}$ (for $k\in\N$) are pairwise disjoint. For every pair 
$\la k,n\ra\in\mathbb{N}^2$, choose a sequence $(y^{(k,n)}_m)_m$ of reals such that $y^{(k,n)}_m\searrow 0$ for all $\la k,n\ra$, and sets of terms of all sequences
$(y^{(k,n)}_m)_m$ are pairwise disjoint, cf. the construction from  Example~\ref{EX1}. For every $k$, let 
$$A_k:=\{ \la x_n^k,y^{(k,n)}_m\ra \colon n,m\in\mathbb{N}\}\cup\{\la a_k,b_k\ra\}$$
 and  $f_k:=k\chi_{A_k}$.  Finally, let $f:=\sum_{k=1}^\infty f_k$. 
 
 First, notice that $f$ is well defined because, for every $\la x,y\ra\in\R^2$, we have $f_k(x,y)\ne 0$ for at most one $k$. Moreover, every $f_k$ has only countable many points of discontinuity, so it is Baire 1. Therefore, $f$ is Baire 2.
 Since $f_k$ is very feebly continuous at $\la a_k,b_k\ra$, $f$ has the same property. Finally, the proof that $f$ is feebly continuous at no point $\la a_k,b_k\ra$, is analogous to the argument in Example \ref{EX1}.
 
(2) Let $C\subset [0,1]$ denote the Cantor ternary set, and let $(I_k)$ be a one-to-one sequence of all components of 
$[0,\infty)\setminus C$ with $I_k=(a_k,b_k)$ for any $k\in\mathbb{N}$. For every $k$, choose a sequence $x^k_n\searrow a_k$ of points from $I_k$. Now, for $k,n,m\in\mathbb{N}$, pick a sequence $(y^{(k,n)}_m)_m$ such that
\begin{itemize}
\item
%$r_{(k,n)}\in (0,\frac{1}{k+n})$;
%\item
$y^{(k,n)}_m\searrow 0$ for every pair $\la k,n\ra$;
\item
sets of terms of $(y^{(k,n)}_m)_m$ are pairwise disjoint.
%if $\la k,n\ra\ne\la k',n'\ra$ then  $\frac{r_{(k,n)}}{r_{(k',n')}}\not\in\mathbb{Q}$.
\end{itemize}
%For any $k,n,m\in\mathbb{N}$, define $y^{(k,n)}_m:=\frac{r_{(k,n)}}{m}$. 
Let $A:=(C\times\{ 0\})\cup \{ \la x^k_n,y_m^{(k,n)}\ra\colon k,n,m\in\mathbb{N}\}$, and  $f:=\chi_A$.
Similarly as in Example \ref{EX1}, one can check that $f$ is Baire 1, $C\times\{ 0\}\subset VFC(f)$, and $f$ is feebly continuous at no point of $C\times\{ 0\}$.

(3) Let $\{ I_k\times J_k\colon k\in\mathbb{N}\}$ be a countable basis of $\R^2$. Let $(C_k)$ be a sequence of pairwise disjoint Cantor sets with $C_k\subset I_k$, and let
 $(d_k)$ be a one-to-one sequence with $d_k\in J_k$ for $k\in\N$. For every $k$, let $f_k\colon\R^2\to \{0,1\}$ be a function constructed as  in (2) with $VFC(f_k)\setminus FC(f_k)=C_k\times\{d_k\}$. Then the function $f:=\sum_{k=1}^\infty kf_k$ is as we need.
\end{proof}

\section{Two-feebly continuity}
In a natural way, one can introduce the notion of reverse very feebly continuity which is different from very feebly
continuity since Example~\ref{exa1} again works. 
Evidently, if a function $f\colon\R^2\to\R$ is (very) feebly continuous and reverse (very) feebly continuous at $\la x,y\ra$, then it has the same
property at $\la y,x\ra$.
We propose another related notion which is stable with respect to the change of order $\la x,y\ra\mapsto \la y,x\ra$. 
We will say that $f\colon\R^2\to\R$ is {\em two-feebly continuous at} $\la x,y\ra$ if there exist sequences $x_n\searrow x$ and $y_n\searrow y$ such that $\lim_n f(x_n,y_n)=f(x,y)$. 
Let us compare the strength of these notions.

\begin{proposition}
If a function $f\colon\R^2\to\R$ is (reverse) very feebly continuous at $\la x,y\ra$ then it is two-feebly continuous at $\la x,y\ra$. The converse need not hold.
\end{proposition}
\begin{proof}
Pick a sequence $x_n\searrow x$ and, for each $n\in\N$, a sequence $y_m^{(n)}\searrow y$, witnessing that $f$ is very feebly continuous at $\la x,y\ra$.
For every $n$, let $\alpha_n:=\lim_m f(x_n,y_m^{(n)})$. Then choose inductively consecutive terms of a sequence $(y_{k_n})$
in such a way that for each $n$, 
\begin{itemize}
\item $y_{k_n}$ is taken from $\{ y_m^{(n)}\colon m\in\N\}$;
\item $k_n<k_{n+1}$;
\item $|f(x_n,y_{k_n})-\alpha_n|<1/n$.
\end{itemize}
Since $\lim_n\alpha_n= f(x,y)$, we have $\lim_n f(x_n,y_{k_n})=f(x,y)$.
Consequently, $f$ is two-feebly continuous at $\la x,y\ra$.
Now, let $A:=\{ \la 1/n,1/n\ra\colon n\in\N\}\cup\{\la 0,0\ra\}$. Clearly, $f:=\chi_{A}$ is two-feebly continuous at $\la 0,0\ra$.
Fix a sequence $x_n\searrow 0$, and, for each $n\in\N$, a sequence $y_m^{(n)}\searrow 0$. 
For every $n$ we have $|\{ y\colon\la x_n,y\ra\in A\}|\leq 1$, so
$\lim_n\lim_m f(x_n,y_m^{(n)})=0\neq 1=f(0,0)$. Hence $f$ is not very feebly continuous at $\la 0,0\ra$.
Similarly, $f$ is not reverse very feebly continuous at $\la 0,0\ra$.
\end{proof}

For a function $f\colon\R^2\to\R$, let $TFC(f)$ denote the set of all points at which $f$ is two-feebly continuous.
\begin{theorem}
For every function $f\colon\R^2\to\R$, the set $TFC(f)$ is residual and has full Lebesgue measure.
If moreover $f$ is Borel, then the set $TFC(f)$ is analytic.
\end{theorem}
\begin{proof}
Observe that $\la x,y\ra$ is in $\R^2\setminus TFC(f)$ if and only if
\begin{eqnarray*}
(\exists p,q\in\Q\,) \left(\la x,y\ra\in f^{-1}[(p,q)]\,\land\,%\\
\exists t\in\Q^+\,\forall s\in(0,t)\, \forall v\in (0,t)\; \langle x+s,y+v\rangle\in f^{-1}[\R\setminus (p,q)]\right).
\end{eqnarray*}
Hence, if $f$ is Borel, the set $\R^2\setminus TFC(f)$ is co-analytic, so $TFC(f)$ is analytic.

For each triple $\la p,q,t\ra\in\Q\times \Q\times\Q^+$, define 
$$ A_{p,q,t}:=\left\{ \la x,y\ra\in f^{-1}[(p,q)]\colon (x,x+t)\times (y,y+t)\subset f^{-1}[\R\setminus (p,q)]\right\}.$$
Then we have 
$$\R^2\setminus TFC(f)=\bigcup_{p,q\in\Q}\bigcup_{t\in\Q^+}A_{p,q,t}.$$
Thus it is enough to verify that each set $A_{p,q,t}$ is nowhere dense and has measure zero.
Clearly, we can assume that $p<q$ and $A_{p,q,t}\ne\emptyset$. Fix $\la x,y\ra\in A_{p,q,t}$ and its open neighbourhood  $(a,b)\times (c,d)$. 
We may assume that $\max(b-a,d-c)<t$. Then the open set $(x,b)\times (y,d)$ is disjoint from $A_{p,q,t}$. This yields that $A_{p,q,t}$ is nowhere dense. 

Let $\lambda_\ast$ (respectively, $\lambda^\ast$) denote inner (outer) Lebesgue measure on $\R^2$. Now, suppose that $\lambda^\ast(A_{p,q,t})>0$. Then there is $\delta\in(0,t)$ such that 
$\lambda^\ast(A_{p,q,t}\cap B)>\frac{3}{4}\lambda(B)$
for an open square $B:=(a,a+\delta)\times (c,c+\delta)$. Consequently, there is $\la x,y\ra\in A_{p,q,t}\cap( (a,a+\frac{\delta}{2})\times (c,c+\frac{\delta}{2}))$. But then $A_{p,q,t}\cap ((x,x+t)\times (y,y+t))=\emptyset$ and $\lambda([(x,x+t)\times(y,y+t)]\cap B)>\frac{1}{4}\lambda(B)$. Therefore, $\lambda_\ast(B\setminus A_{p,q,t})>\frac{1}{4}\lambda(B)$, contrary to the choice of $B$.
\end{proof}

\begin{remark}
Let us remark that the set $\R^2\setminus TFC(f)$ may be of cardinality $\mathfrak{c}$, even for Baire 1 functions. Consider for instance, $f:=\chi_{\R\times\{ 0\}}$. Then $TFC(f)=\R^2\setminus(\R\times\{ 0\})$. In general, the set $TFC(f)$ may be non-analytic: let $f:=\chi_{B\times\{0\}}$ where $B\subset\R$ is not co-analytic.
\end{remark}

%%%%%%%%%%%%%%%%%%%%%%%%%%%%%%%%%%%%%%%%%%%%%%%%%%%%%%%%%%%%%%%%%%%%%%%%%%%%%%%%%%%%%%%%%%%%%%%%%%%%%%%%%%%%%%%%%%%%%%%%%%%%%%%%%%
\section{Large algebraic structures within families of feebly continuous functions}
Since the first decade of this century, extensive investigations of large algebraic structures within various sets of exotic  functions or sequences
has been conducted by several reaserchers starting from V.~I.~Gurariy, R.~M.~Aron, J.~B.~Soane Sep\'ulveda and others.
In particular, the notions of lineability and algebraibility have been introduced and studied.
See the recent survey \cite{BPS} and monograph \cite{ABPS}. Let us recall basic definitions in this topic.
We will use them considering some families of feebly continuous functions.

Let $\kappa$ be a cardinal number. 
%\begin{itemize}
%\item Let $L$ be a vector space. We say that a set $A\subset L$ is {\em $\kappa$-lineable}
%if $A\cup\{0\}$ contains a $\kappa$-dimensional vector space.
%\item 
Let $L$ be a linear commutative algebra. We say that a set $A\subset L$ is {\em $\kappa$-algebrable} if $A\cup\{0\}$ contains a $\kappa$-generated algebra $B$. We say that $A\subset L$ is {\em strongly $\kappa$-algebrable} if 
$A\cup\{0\}$ contains a $\kappa$-generated algebra that is isomorphic to a free algebra 
(cf. \cite{BG}).
%\end{itemize}
Note that $X=\{ x_\alpha\colon \alpha<\kappa\}$ is a set of generators of a free algebra 
contained in $A\cup\{0\}$ if and only if, for any $k\in\N$ any nonzero polynomial $P$  in $k$ variables, without constant term, and any distinct 
$y_1,\dots ,y_k\in X$, we have $P(y_1,\dots ,y_k)\neq 0$ and $P(y_1,\dots, y_k)\in A$.

Clearly, if $L$ is a linear commutative algebra $L$, then the $\kappa$-strong algebrability of $A\subset L$ implies its 
$\kappa$-algebrability, and this in turn implies the $\kappa$-linearity of $A$.  
We will consider various families connected with feebly continuity, contained in the linear commutative algebra $\R^{\R^2}$.

\begin{lemma}\label{lem:lineability}
For any function $f\colon\R\to\R$, let $\widetilde{f}\colon\R^2\to\R$ be defined by $\widetilde{f}(x,y):=f(x)$. Then $\widetilde{f}$ is feebly continuous if and only if $f$ has the same property.
\end{lemma}
\begin{proof}
Assume $f$ is feebly continuous. Fix $\la x,y\ra\in\R^2$. Since $f$ is feebly continuous at $x$, there is a sequence $x_n\searrow x$ with $\lim_nf(x_n)=f(x)$. Fix any sequence $y_m\searrow y$. Then $\lim_n\lim_m \widetilde{f}(x_n,y_m)=\lim_n f(x_n)=\widetilde{f}(x,y)$, so $\widetilde{f}$ is feebly continuous at the point $\la x,y\ra$.

Now, assume that $\widetilde{f}$ is feebly continuous. Fix $x\in\R$. Since $\widetilde{f}$ is feebly continuous at the point $\la x,0\ra$, there are sequences $x_n\searrow x$, $y_m\searrow 0$ with $\lim_n\lim_m\widetilde{f}(x_n,y_m)=\widetilde{f}(x,0)$. Then for each $n$, $\lim_m\widetilde{f}(x_n,y_m)=\lim_m f(x_n)=f(x_n)$ and $\widetilde{f}(x,0)=f(x)$, so $\lim_n f(x_n)=f(x)$, which means that $f$ is feebly continuous at $x$.
\end{proof}

We will denote by $\FC$ (respectively, $\VFC$, $\TFC$) the sets of all feebly continuous (respectively,
very feebly continuous, two-feebly continuous) functions $f\colon\R^2\to\R$.
Additionally, let $\FC(\R)$ stand for the family of all feebly continuous functions $f\colon\R\to\R$.
We know that $\FC\subset\VFC\subset\TFC$. We will show that the families $\FC$, $\VFC\setminus\FC$ and $\TFC\setminus\VFC$ are large from the algebraic viewpoint.

\begin{theorem}
The families $\FC$ and $\TFC\setminus\VFC$ are strongly $2^{\co}$-algebrable.
\end{theorem}
\begin{proof}
It is known that the family $\ES^+(\R)$, of all functions $f\colon\R\to[0,\infty)$ with dense level sets, is strongly $2^\co$-algebrable, see \cite{BGP}. It is easy to check that every function $f\in\ES^+(\R)$ belongs to $\FC(\R)$.
 Let $\{ f_\xi\colon 
\xi<2^\co\}$ be a $2^\co$-generated algebra, contained in $\FC(\R)$ and isomorphic to a free algebra. For every $\xi< 2^\co$, let $\widetilde{f}_\xi\colon\R^2\to\R$ be defined by $\widetilde{f}_\xi(x,y):=f_\xi(x)$. Clearly, $\widetilde{f}_\xi$'s are pairwise different. By Lemma \ref{lem:lineability}, all $\widetilde{f}_\xi$'s are feebly contiunous. Moreover, it is 
easy to verify that the set $\{\widetilde{f}_\xi\colon \xi<2^\co\}$ is a $2^\co$-generated algebra isomorphic to a free algebra.
Hence $\FC$ is strongly $2^{\co}$-algebrable.

The proof for $\TFC\setminus\VFC$ is similar. Observe that, if $f\in\ES^+(\R)$, then the function $F\colon \R^2\to\R$ given by 
$F(x,y):=f(x)$ whenever $x=y$, and $F(x,y):=0$, otherwise, belongs to $\TFC\setminus\VFC$. The remaining argument works as above.
\end{proof}  

%%%%%%%%%%%%%%%%%%%%%%%%%%%%%%%%%%%%%%%%%%%%%%%%%%%%%%%%%%
Now, we will study the algebrability of the family $\VFC\setminus\FC$.
We apply the method using ultrafilters on $\N$, initiated in \cite{BBGN} and developed in \cite{CRS}.
Define the standard projections $\pi_i\colon\R^2\to\R$ ($i=1,2$) by
$\pi_1(x,y):=x$, $\pi_2(x,y):=y$.

Let $J\subset\R$ be a non-degenerate interval. We say that a set $D\subset\R^2$ is $(\VFC,J)$-{\em massive} provided there exists a map $f\colon D\to \int(J)$ such that each function $g\colon\R^2\to J$, which is equal to $f$ on $D$, 
belongs to $\VFC$. (This notion mimics a similar idea in \cite{N}.)

We say that $D\subset\R^2$ is a D-{\em set} if it has the following properties:
\begin{itemize}
\item[(i)]
the set $\pi_1[D]$ is countable and dense in $\R$;
\item[(ii)]
for each $x\in\pi_1[D]$, the $x$-section $D_x:=\{ y\in\R\colon \la x,y\ra\in D\}$ of $D$ is countable dense in $\R$;
\item[(iii)]
the $x$-sections of $D$ are pairwise disjoint.
\end{itemize}

\begin{lemma} \label{LL}
There exists a family $\D$ of $\co$-many  {\em D}-sets contained in $\R^2$ such that:
\begin{enumerate}
\item
$\pi_2[D]\cap\pi_2[D']=\emptyset$ for any distinct $D, D'\in\D$;
\item
 for every $D\in\D$, $D$ is $(\VFC,J)$-massive for any non-degenerate interval $J$.
 \end{enumerate}
\end{lemma}
\begin{proof}
Decompose $\R$ into $\co$-many countable dense sets $\{ C_\alpha \colon \alpha<\co\}$. For every $\alpha<\co$, decompose $C_\alpha$ into infinitely many dense sets $\{ C_{\alpha,n}\colon n\in\N\}$. List elements of $C_\alpha$ as a one-to-one sequence
$(c_n^\alpha)_{n\in\N}$
and let $D_\alpha:=\bigcup_{n\in\N}\{c^\alpha_n\}\times C_{\alpha,n}$ for $\alpha<\co$. Clearly, each $D_\alpha$, for $\alpha<\co$, is a D-set. Define 
$\D:=\{ D_\alpha\colon \alpha<\co\}$. Then condition (1) is fulfilled.

To show (2), fix a non-degenerate interval $J\subset\R$ and $D_\alpha\in\D$. Let $\{ q_n\colon n\in\N\}$ be 
a countable dense subset of $J$. Then the function $f\colon D_\alpha\to J$, defined by $f(x,y):=q_n$, whenever
$\la x,y\ra\in D_\alpha$ with $x\in C_{\alpha,n}$, witnesses that $D_\alpha$ is $(\VFC,J)$-massive.
Indeed, consider an extension $g\colon\R^2\to J$ of $f$ and 
fix $\la x,y\ra\in\R^2$. Let $(q_{k_n})$ be a subsequence of $(q_n)$ convergent to $g(x,y)$. 
Choose a sequence $x_n\searrow x$ with $x\in C_{\alpha, k_n}$ for every $n\in\N$. Thus
for each $n\in\N$, we have $x=c^\alpha_{j_n}$ for a unique $j_n\in\N$. Now, for each $n\in\N$, pick a sequence 
$(y_m^{(n)})_m$ with terms in $C_{\alpha,j_n}$ and such that $y_m^{(n)}\searrow y$. Then $\la x_n,y_m^{(n)}\ra\in D_\alpha$
for all $n,m\in\N$, so the vaules of $f$ and $g$ are equal at these points.
For each $n\in\N$, we have $\lim_m f(x_n,y_m^{(n)})=\lim_m q_{k_n}=q_{k_n}$. 
Hence $\lim_n\lim_m g(x_n,y_m^{(n)})=\lim_n q_{k_n}=g(x,y)$. Thus $g$ is very feebly continuous at $\la x,y\ra$.
\end{proof}

Let $\H^n$ denote the family of all polynomials  from $\R^n$ to $\R$ without constant term.

\begin{theorem}
There exists a family $\F\subset\R^{\R^2}$ of cardinality $2^\co$ such that, for each $h\in\H^n$ and every sequence $\la f_1,\ldots,f_n\ra$ with terms in $\F$, we have $h(f_1,\ldots,f_n)\in\VFC\setminus \FC$. Consequently, the family $\VFC\setminus\FC$ is strongly $2^\co$-algebrable.
\end{theorem}
\begin{proof}
We use the ideas from the proof of \cite[Theorem 3.1]{CRS}. Let $\H:=\bigcup_{n\in\N}\H^n\times \{1,\dots ,n\}^\N$. List all sets of the family $\D$ from Lemma~\ref{LL} in a one-to-one fashion as $\{ D_{h,p}\colon \la h,p\ra\in\H\}$. For every 
$\la h,p\ra\in\H^n\times \{ 1,\dots ,n\}^\N$, let $g_{h,p}\colon D_{h,p}\to \int(h[\R^n])$ witness the fact that $D_{h,p}$ is $(\VFC,h[\R^n])$-massive. For each $z=\la x,y\ra$, $z\in D_{h,p}$, let $\vec{v}_{h,p}(z)=(\vec{v}_{h,p}(z)_{1},\ldots,\vec{v}_{h,p}(z)_{n})\in\R^n$ be such that $h(\vec{v}_{h,p}(z))=g_{h,p}(z)$. Moreover, we put $\vec{v}_{h,p}(z):=0$ for $z\in\R^2\setminus D_{h,p}$. 

Let $\bar{p}\colon\beta\N\to\{ 1,\dots ,n\}$ be a continuous extension of $p$ to $\beta\N$, the \v{C}ech-Stone compactification of $\N$. For every ultrafilter 
$\U\in\beta\N$, we define the function $f_\U\colon\R^2\to\R$ as follows. If $z\in D_{h,p}$ for some $\la h,p\ra\in\H$ 
then $f_\U(z):=\vec{v}_{h,p}(z)_{\bar{p}(\U)}$. Otherwise, let $f_\U(z):=0$.

Then the family $\F:=\{ f_\U\colon\U\in\beta\N\}$ is as we need. Indeed, fix  distinct $\U_1,\ldots,\U_{n}\in\beta\N$ and an $h\in\H^n$. 
Then $f_{\U_1},\ldots,f_{\U_{n}}$ are distinct, see the last part of the proof of \cite[Theorem 3.1]{CRS}.

Let $f:=h(f_{\U_1},\ldots,f_{\U_{n}})$ and $J:=h[\R^n]$. The proof of the fact that $f\in\VFC$ is exactly the same as in \cite[Theorem 3.1]{CRS}. Let us give some details for the reader's convenience. Choose a partition $\{U_1,\dots ,U_n\}$ of $\N$ such that $U_i\in \U_j$ if and only if $i=j$. Let $p\in\{ 1,\dots n\}^\N$ be such that $p^{-1}[\{ i\}]=U_i$ for each $i\in\{ 1,\dots ,n\}$.
Consider $\bar{p}\colon\beta\N\to\{ 1,\dots ,n\}$ and observe that $\bar{p}(\U_i)=i$ for every $i\in\{ 1,\dots ,n\}$.
Then one can check that, for each $z\in D_{h,p}$,
$$f(z)=h(\vec{v}_{h,p}(z))=g_{h,p}(z).$$
Since $D_{h,p}$ is $(\VFC, J)$-massive, the extension $f$ of $g_{h,p}$ to $\R^2$ is in $\VFC$ as desired.

Moreover, observe that if $x\ne x'$ then there is no $y\in\R$ such that $f(x,y)\ne 0\ne f(x',y)$. 
Indeed, if  $f(x,y)\ne 0$ then $f_{\U_i}(x,y)\ne 0$ for some $i\le n$, and this implies that $\la x,y\ra\in D_{t,p}$ for some $\la t,p\ra\in\H$.
Thus, if $f(x,y)\ne 0\ne f(x',y)$ then $\la x,y\ra\in D_{t,p}$ and $\la x',y\ra\in D_{t',p'}$ for some $\la t,p\ra, \la t',p'\ra\in\H$. By condition (1) in Lemma~\ref{LL}, the pairs $\la t,p\ra$ and  $\la t',p'\ra$ can not be distinct. Hence $\la t,p\ra=\la t',p'\ra$ and then the condition (iii) yields $x=x'$.
 This shows that 
$f$ is not feebly continuous at any point $\la x,y\ra\in \R^2\setminus f^{-1}[\{0\}]$. Hence $f\notin\FC$.

The proof that $\F$ is of cardinality $2^\co$ is given in \cite{CRS}; it follows from $|\beta\N|=2^\co$ and the fact that all maps $f_\U$ are distinct.
Finally, the family $\F$ witnesses that $\VFC\setminus \FC$ is strongly $2^\co$-algebrable. To see it,  we need to observe that $h(f_{\U_1},\ldots,f_{\U_{n}})\ne 0$ for all $h\in\H$ and $\U_1,\ldots \U_n\in\beta\N$; cf.
the final part of the proof of Theorem~3.1  in \cite{CRS}.
\end{proof}

%\begin{corollary}
%The family $\VFC\setminus\FC$ is strongly $2^\co$-algebrable.
%\end{corollary}

%%%%%%%%%%%%%%%%%%%%%%%%%%%%%%%%%%%%%%%%%%%%%%%%%%%%%%%%%%%%%%%%%%%%%

As it was mentioned in Section 2, Leader constructed, under {\bf CH}, a function $f\colon\R^2\to\R$ which is nowhere
feebly continuous, that is, $FC(f)=\emptyset$. We do not know whether nowhere feebly continuous functions can exist in models of ZFC without {\bf CH}. 

Below, we will modify a bit Leader's construction. In fact we prove that the function $g:=|f|$, where $f$ comes from Leader's construction, is also nowhere feebly continuous.
Moreover, the range of $g$ equals $\{0\}\cup\N$, and
$g$ is symmetric, that is $g(x,y)=g(y,x)$ for all $x,y\in\R$.
The function $g$ will help us to infer that, under {\bf CH}, the set of nowhere feebly continuous functions is strongly $\co$-algebrable.
We do not know whether this result can be improved to obtain strong $2^\co$-algebrability. 
%Also, we don't know anything about algebrability the family of nowhere feebly continuous functions in models of ZFC without CH. We do not even know if nowhere feebly continuous functions exist in such models.

\begin{proposition} \label{lepszyLeader}
Assume {\bf CH}. There exists a symmetric surjection $g\colon\R^2\to\{0\}\cup \N$ that is nowhere feebly continuous.
Additionally, assume that $h\colon\R\to\R$ is a continuous function, being strictly monotone on every set from 
a finite partition of $\R$ into intervals, with $\lim_{x\to\infty}h(x)=\pm\infty$. Then $h\circ g$ is nowhere feebly continuous.
\end{proposition}
\begin{proof}
Let $\{ r_\alpha\colon\alpha<\omega_1\}$ be a one-to-one enumeration of $\R$, and for each $\alpha<\omega_1$, let $\{r^\alpha_n\colon n<\omega\}$ be a one-to-one enumeration of the set $\{ r_\beta\colon\beta<\alpha\}$. Then define $g(x,y):=0$ whenever $x=y$. Otherwise $\{ x,y\}=\{r_\alpha, r_\beta\}$ for some $\beta<\alpha<\omega_1$ and then $r_\beta=r^\alpha_k$ for exactly one  $k\in\N$. Then put $g(x,y):=k$. Obviously, $g$ is a symmetric function from $\R^2$ onto $\{0\}\cup\N$. 

We will verify that $g$ is feebly continuous at no point $\la x,y\ra$. Clearly, $g$ is feebly continuous at no point 
of the form $\la x,x\ra$. Now, fix $\la x,y\ra\in\R^2$ with $x\ne y$, and let $g(x,y)=k>0$. Suppose $g$ is feebly continuous at $\la x,y\ra$, so there are sequences $x_n\searrow x$, $y_m\searrow y$ with $\lim_n\lim_m g(x_n,y_m)=k$. We may assume that $\lim_mg(x_n,y_m)=k$ for every $n\in\N$, hence for every $n$ there is $m_n$ such that $g(x_n,y_m)=k$ for every $m>m_n$. Fix $n$, and let $x_n=r_\alpha$. Observe that there is at most one $m>m_n$ such that $y_m=r_\beta$ where $\beta<\alpha$. Increasing $m_n$ as needed, we can assume that 
if $m>m_n$ then $y_m=r_\beta$ with $\alpha<\beta$. Now, fix $m>\max(m_1,m_2)$. Let $x_1=r_{\alpha_1}$, $x_2=r_{\alpha_2}$, $y_m=r_\beta$. Then $\alpha_1<\beta$ and  $\alpha_2<\beta$. Since $g(x_1,y_m)=k$, we have $r_{\alpha_1}=r^\beta_k=r_{\alpha_2}$. Hence $x_1=x_2$ which yields a contradiction.

Let us prove the second assertion. From the assumptions it follows that $h$ is finite-to-one and it maps 
$\{0\}\cup\N=g[\R^2]$ onto a discrete set. To show that $h\circ g$ is nowhere feebly continuous, we modify the above argument.
Since $h[\{0\}\cup\N]$ is a discrete set, it follows that $h\circ g$ is feebly continuous at no point 
$\la x,x\ra$. 

Now, fix $\la x,y\ra$ with $x\neq y$ and let $g(x,y)=k>0$.
Suppose $h\circ g$ is feebly continuous at $\la x,y\ra$, so there are sequences $x_n\searrow x$, $y_m\searrow y$ with 
$\lim_n\lim_m (h\circ g)(x_n,y_m)=h(k)$. We may assume that $\lim_m(h\circ g)(x_n,y_m)=h(k)$ for every $n\in\N$.
Hence for every $n$ there is $m_n$ such that $(h\circ g)(x_n,y_m)=h(k)$ for every $m>m_n$.
Since $h$ is finite-to-one, assume that $h^{-1}[h[\{k\}]]\cap\N=\{k_1,\dots,k_p\}$. 
Thus
$g(x_n,y_m)\in \{k_1,\dots,k_p\}$ for all $n$ and $m>m_n$. Fix $n$, and let $x_n=r_\alpha$. 
Observe that there is at most one $m>m_n$ such that $y_m=r_\beta$ where $\beta<\alpha$. Increasing $m_n$ as needed, we can assume that 
if $m>m_n$ then $y_m=r_\beta$ with $\alpha<\beta$.

Fix $m>\max(m_1,\dots , m_{p+1})$. Let $x_i=r_{\alpha_i}$ for $i=1,\dots ,p+1$, and let $y_m=r_\beta$. 
Then $\alpha_i<\beta$ for $i=1,\dots ,p+1$. We have $g(x_i,y_m)\in\{k_1,\dots, k_p\}$ for $i=1,\dots, p+1$.
By the pigeon hole principle, we can find $j\in\{ 1,\dots ,p\}$ and two distinct indices $i_1,i_2\in\{ 1,\dots, p+1\}$ 
such that  $g(x_{i_1},y_m)=k_j=g(x_{i_2},y_m)$. Then by the definition of $g$,
we have $r_{\alpha_{i_1}}=r^\beta_{k_j}=r_{\alpha_{i_2}}$. Hence $x_{i_1}=x_{i_2}$ which yields a~contradiction.
\end{proof}

\begin{theorem} \label{ndc}
Assume {\bf CH}. Then the set of nowhere feebly continuous functions from $\R^2$ to $\R$ is strongly $\co$-algebrable.
\end{theorem}
\begin{proof} 
We use the method proposed in \cite[Proposition~7]{BBF}. Consider the function $g$ from Proposition~\ref{lepszyLeader}.
It suffices to show that $h\circ g$
is nowhere feebly continuous for every exponential like function $h\colon\R\to\R$ of the form
$$h(x):=\sum_{i=1}^m a_ie^{\beta_i x}$$
where $\beta_i>0$ for $i=1,\dots ,m$. Then functions of the form $x\mapsto e^{\beta g(x)}$, where parameters $\beta$ are taken from a Hamel basis (of $\R$ over $\Q$), are free generators of an algebra
included in the set of 
%nowhere feebly 
continuous functions.
It is easy to check that $h$ given by the above formula satisfies the conditions from the second part of Proposition~\ref{lepszyLeader};
cf. \cite[Lemma 8]{BBF}. Hence we obtain the assertion.
\end{proof}


\begin{thebibliography}{abc}
\bibitem{ABPS} R. M. Aron, L. Bernal Gonz\'alez, D. M. Pelegrino, J. B. Seoane Sep\'ulveda, {\it Lineability: The Search of Lineability in Mathematics}, Monographs and Research Notes in Mathematics, CRC Press, Boca Raton, FL, 2016.
\bibitem{BBF} M. Balcerzak, A. Bartoszewicz, M. Filipczak, {\it Nonseparable spaceability and strong algebrability
of sets of continuous singular functions}, J. Math. Anal. Appl. {\bf 407} (2013), 263--269.
\bibitem{BBGN} A. Bartoszewicz, M. Bienias, Sz. G{\l}\c{a}b, T. Natkaniec, {\it Algebraic structures in the sets of surjective functions}, J. Math. Anal. Appl. {\bf 441} (2016), 574--585.
\bibitem{BG} A. Bartoszewicz, Sz. G{\l}\c{a}b, {\it Strong algebrability of sequences and functions}, Proc. Amer. Math. Soc. {\bf 141} (2013), 827--835.
\bibitem{BGPS} A. Bartoszewicz, Sz. G{\l}\c{a}b, D. Pelegrino, J. B. Seoane Sep\'ulveda, {\it Algebrability, non-linear properties and special functions}, Proc. Amer. Math. Soc. {\bf 141} (2013), 3391-3402.
\bibitem{BGP} A. Bartoszewicz, Sz. G{\l}\c{a}b, A. Paszkiewicz, {\it Large free linear algebras of real and complex functions}, Linear Algebra Appl. {\bf 438} (2013), 3689--3701.
\bibitem{BPS} L. Bernal Gonz\'alez, D. M. Pelegrino, J. B. Seoane Sep\'ulveda, {\it Linear subsets of nonlinear sets in topological vector spaces}, Bull. Amer. Math. Soc. {\bf 51} (2014), 71--130.
\bibitem{Br} M. L. Brodskii, {\it On some properties of sets of positive measure}, Uspekhi Mat. Nauk. {\bf 31} (1949), 136--139 (in Russian).
\bibitem{CRS} K. C. Ciesielski, D. L.Rodr\'{i}guez-Vidanes, J. B. Seoane Sep\'ulveda, {\it Algebras of measurable extendable functions of maximal cardinality}, Linear Algebra Appl. {\bf 565} (2019), 258--266.
\bibitem{C} E. F. Collingwood, {\it Cluster sets of arbitrary functions}, Proc. Nat. Acad. Sci. USA {\bf 46} (1960), 1236--1242.
\bibitem{E} H. G. Eggleston, {\it Two measure properties of Cartesian product sets}, Quart. J. Math.,
Oxford Ser. (2) {\bf 5} (1954), 108--115.
\bibitem{GMSS} J. L. G\'{a}mez-Merino, G. A. Mu\~{n}oz-Fern\'{a}ndez, V. M.~S\'{a}nchez, J. B.~Seoane-Sep\'{u}lveda, \emph{Sierpi\'{n}ski-Zygmund functions and other problems on lineability}, Proc. Amer. Math. Soc. \textbf{138} (2010), 3863--3876.
\bibitem{Ke} A. S. Kechris, {\it Classical Descriptive Set Theory}, Springer, New York 1995.
%\bibitem{K} K. Kuratowski, {\it Topology}, vol. 1, Academic Press, New York 1966.
\bibitem{L} I. Leader, {\it A note on feebly continuous functions}, Topology Appl. {\bf 156} (2009), 2629--2631.
\bibitem{N} T. Natkaniec, {\it Algebrability of  some families of Darboux-like functions}, Linear Algebra Appl. {\bf 439} (2013), 3256--3263.
\bibitem{O} J. C. Oxtoby, {\it Measure and Category}, (2nd edition), Springer, New York 1980.
\bibitem{S} S. M. Srivastava, {\it A Course of Borel Sets}, Springer, New York 1998.
\bibitem{T} B. S. Thomson, {\it Real functions}, Lecture Notes in Mathematics 1170, Springer, New York 1985.
\end{thebibliography}
\end{document}